\documentclass{amsart}
\usepackage{bullclass}

\usepackage{amsmath,cite,graphicx}

\usepackage{amsthm,amssymb}

\newcommand{\A}{\mathcal{A}}
\DeclareMathOperator{\RE}{Re}

\begin{document}

%\bullart{34}{3}{2011}{611}{629}

\title[Differential Subordination
for Analytic Functions  with Fixed Initial Coefficient]{
Second-Order Differential Subordination  for Analytic \\Functions
with Fixed Initial Coefficient }

\author[R. M. Ali, S. Nagpal and V. Ravichandran]{\first Rosihan M. Ali,
\second Sumit Nagpal and \third V. Ravichandran}

\address{$^{1,3}$ School of Mathematical Sciences,
Universiti Sains Malaysia, 11800 USM Penang, Malaysia
\\  $^{2,3}$ Department of Mathematics, University of Delhi,
Delhi 110 007, India\\}

\email{\first rosihan@cs.usm.my, \second
sumitnagpal.du@gmail.com,\third vravi68@gmail.com}

\begin{abstract}
Functions with fixed initial coefficient  have been widely studied.
A new methodology is proposed in this paper by making appropriate
modifications and improvements to the theory of second-order
differential subordination. Several interesting examples are given.
The results obtained are applied to the classes of convex and
starlike functions with fixed second coefficient.
\end{abstract}



%\recdate{January 3, 2011}{March 2, 2011}{}{Lee See Keong}

\maketitle

\section{Introduction}\label{ch3,sec1} \noindent
Let $\mathcal{H}$ denote the class of analytic functions in the unit
disk $\mathbb{D}:=\{z\in\mathbb{C}:|z|<1\}$. For a fixed positive
integer $n$, let $\mathcal{H}[a,n]$ be its subset consisting of
functions $p$ of the form \[p(z)=a+p_n z^{n}+p_{n+1}z^{n+1}+\cdots.
\] The familiar subclass $\mathcal{S}$ of $\mathcal{H}[0,1]$
consists of normalized univalent functions
$f(z)=z+a_{2}z^{2}+a_{3}z^{3}+\cdots$ in $\mathbb{D}$. It is a
remarkable fact that the second coefficient plays an important role
in univalent function theory; indeed it influences growth and
distortion estimates \cite{duren} for functions in the class
$\mathcal{S}$. For this reason, there is a continued interest in the
investigations of how the second coefficient shaped the geometric
properties of important subclasses of functions. Works in this
direction include those of
\cite{rma,rma2,finkelstein,gronwall,ksp2,tepper}.

This paper studies further analytic functions with fixed initial
coefficient. The methodology used here will be differential
subordination, applied from making appropriate modifications and
improvements to its existing theory. Many of the significant works
on differential subordination have been pioneered by Miller and
Mocanu \cite{millermocanu1978,millermocanu1981}, and their monograph
\cite{monograph} compiled their great efforts over more than two
decades. In recent years, various authors have successfully applied
the theory of differential subordination to address many important
problems in the field. These works include those of
\cite{ros1,ros2,ros3,ros4,ros5,millerthird,cho,wang}.

To state the central theme of Miller and Mocanu's theory on
differential subordination, let $p$ be an analytic function in
$\mathbb{D}$ and $\psi(r,s,t)$ be a complex function defined in a
domain of $\mathbb{C}^{3}$. Consider a class of functions $\Psi$,
and two subsets $\Omega$ and $\Delta$ in $\mathbb{C}$. Given any two
quantities, the aim of the theory of second-order differential
subordination is to determine the third so that the following
differential implication is satisfied: \begin{equation*}\label{subo}
\psi \in \Psi \quad \mbox{and} \quad \{\psi(p(z),z
p'(z),z^{2}p''(z)): z \in \mathbb{D}\} \subset \Omega \quad
\Rightarrow \quad p(\mathbb{D}) \subset \Delta. \end{equation*}

Let $\mathcal{H}_{\beta}[a,n]$ consist of analytic functions $p$ of
the form
\[p(z)=a+p_n z^{n}+p_{n+1}z^{n+1}+\cdots,\]
where the coefficient $p_{n}$ is a fixed constant $\beta$ in
$\mathbb{C}$. Without loss of generality, $\beta$ is assumed to be a
nonnegative real number.

Let $\psi:\mathbb{C}^{3}\rightarrow \mathbb{C}$ be analytic in a
domain $D$ and let $h$ be univalent in $\mathbb{D}$. Suppose $p\in
\mathcal{H}_{\beta}[a,n]$, $(p(z),z p'(z),z^{2}p''(z)) \in D$ when
$z \in \mathbb{D}$, and $p$ satisfies the second-order differential
subordination
\begin{equation}\label{ch3,eq2.4}
\psi(p(z),zp'(z),z^{2}p''(z)) \prec h(z).
\end{equation}
Then $p$ is called a $\beta$-\emph{solution} of the differential
subordination. The univalent function $q$ is said to be a
$\beta$-\emph{dominant} of the differential subordination
\eqref{ch3,eq2.4} if $p \prec q$ for all $p$ satisfying
\eqref{ch3,eq2.4}. If $\widetilde{q}$ is a $\beta$-dominant of
\eqref{ch3,eq2.4} and $\widetilde{q} \prec q$ for all
$\beta$-dominants $q$ of \eqref{ch3,eq2.4}, then $\widetilde{q}$ is
called the $\beta$-\emph{best dominant} of \eqref{ch3,eq2.4}. The
class of functions $\Psi_\beta$ such that \eqref{ch3,eq2.4} is
satisfied for $\psi\in \Psi_\beta$ is called the class of
\emph{$\beta$-admissible functions}.

Section \ref{ch3,sec3} of this paper deals with the basic lemmas for
functions with fixed initial coefficient. In Section \ref{ch3,sec4}
a suitable class $\Psi_{n,\beta}(\Omega,q)$ of $\beta$-admissible
functions is defined and theorems analogous to those of Miller and
Mocanu \cite{monograph} are obtained. These results are applied to
two important particular cases corresponding to $q(\mathbb{D})$
being a disk or a half-plane. Examples of differential inequalities
and subordinations are presented in Section \ref{ch3,sec5}. The
paper concludes with interesting applications of the newly
formulated theory to the classes of normalized convex and starlike
univalent functions with fixed second
coefficient.

The following extended version of Schwarz Lemma (see
\cite{finkelstein}) is required in our investigations.

\begin{lemma}[Extended  Schwarz Lemma]\label{efsl} Let
$w(z)=a_{1}z+\cdots$ be an analytic map of the unit disk
$\mathbb{D}$ into itself. Then $|a_{1}| \leq 1$, and \[|w(z)| \leq
\frac{|z|(|z|+|a_{1}|)}{1+|a_{1}| |z|}.\] Equality holds at some
point $z \neq 0$ if and only if
\[w(z)=\frac{e^{-it}z(z+a_{1}e^{it})}{1+\bar{a}_{1}e^{-it}z},\quad t
\geq 0.\]
\end{lemma}

%\section{Subordination Problems with Fixed Initial Coefficient}\label{ch3,sec2}

\section{Fundamental lemmas for functions with
fixed initial coefficient}\label{ch3,sec3} \noindent Let
$\mathbb{D}_r:=\{z\in\mathbb{C}:|z|<r\}$. In this section, we will
prove several basic  lemmas  for functions with fixed initial
coefficient.

\begin{lemma}\label{ch3,lem3.1} Let
$z_{0}=r_{0}e^{i\theta_{0}}$, $(r_{0}<1)$, and
$g(z)=g_{n}z^{n}+g_{n+1}z^{n+1}+\cdots$ be continuous in
$\overline{\mathbb{D}}_{r_0}$, analytic in
$\mathbb{D}_{r_0}\cup\{z_0\}$ with $g(z)\not\equiv 0$, and $n \geq
1$. If
\[|g(z_{0})|=\max_{|z|\leq r_{0}} |g(z)|\]
then
\[\frac{z_{0}g'(z_{0})}{g(z_{0})}=m\quad \text{and}\quad
\RE \left(\frac{z_{0}g''(z_{0})}{g'(z_{0})}+1\right) \geq m,\] where
\begin{equation}\label{ch3,eq3.1}
m \geq
n+\frac{|g(z_{0})|-|g_{n}|r_{0}^{n}}{|g(z_{0})|+|g_{n}|r_{0}^{n}}.
\end{equation}
\end{lemma}

%\begin{lemma}\label{ch3,lem3.1} Let
%$z_{0}=r_{0}e^{i\theta_{0}}$, $(r_{0}<1)$. Let
%$g(z)=g_{n}z^{n}+g_{n+1}z^{n+1}+\cdots$ be a function with the fixed
%coefficient $g_{n}$, continuous in $\overline{\mathbb{D}}_{r_0}$,
%analytic in $\mathbb{D}_{r_0}\cup\{z_0\}$, with $g(z)\not\equiv 0$
%and $n \geq 1$. If
%\[|g(z_{0})|=\max_{|z|\leq r_{0}} |g(z)|\]
%then
%\[\frac{z_{0}g'(z_{0})}{g(z_{0})}=m\quad \text{and}\quad
%\RE \left(\frac{z_{0}g''(z_{0})}{g'(z_{0})}+1\right) \geq m,\]
%where
%\begin{equation}\label{ch3,eq3.1}
%m \geq n+\frac{|g(z_{0})|-|g_{n}|}{|g(z_{0})|+|g_{n}|}.
%\end{equation}
%\end{lemma}

\begin{proof}The first two assertions follow from Lemma 2.2a in
\cite[p.\ 19]{monograph}. It remains to prove \eqref{ch3,eq3.1}. Let
$h:\overline{\mathbb{D}} \rightarrow \mathbb{C}$  be defined by
\begin{align*}
    h(z)&=\frac{g(z_{0}z)}{g(z_{0})z^{n-1}} = s_n z+ \cdots,
    \quad \text{where}\quad   s_{n}=\frac{g_{n}z_{0}^{n}}{g(z_{0})}.
\end{align*}
 Then  $h$ is
continuous in $\overline{\mathbb{D}}$ and analytic in
$\mathbb{D}\cup \{1\}$, and the maximum principle readily gives
\begin{align*}
    |h(z)|\leq \max_{|z|=1} |h(z)|
          =\frac{1}{|g(z_{0})|} \max_{|z|= 1} |g(z_{0}z)|
           = 1.
\end{align*}
Since $h(0)=0$, the extended Schwarz Lemma (Lemma \ref{efsl}) yields
$|s_{n}| \leq 1$, and
\begin{equation*}\label{ch3,eq3.2}
\left|\frac{g(z_{0}z)}{g(z_{0})z^{n-1}}\right| = |h(z)| \leq
\frac{|z|(|z|+|s_{n}|)}{1+|s_{n}||z|}.
\end{equation*}
In particular, at the point $z=r$, $0 \leq r <1$,
\begin{equation}\label{ch3,eq3.3}
 \RE \frac{g(z_{0}r)}{g(z_{0})} \leq
  \left|\frac{g(z_{0}r)}{g(z_{0})}\right| \leq
  \frac{r^{n}(r+|s_{n}|)}{1+|s_{n}|r}.
\end{equation}
Since $m=z_{0}g'(z_{0})/g(z_{0})$,
\begin{align*}
    m =\left.\frac{d}{dr}\left( \frac{g(z_{0}r)}{g(z_{0})}\right)
    \right|_{r=1}
      =\lim_{r\rightarrow
      1}\frac{g(z_{0}r)-g(z_{0})}{(r-1)z_{0}}\frac{z_{0}}{g(z_{0})}
      =\lim_{r\rightarrow 1}
     \left(1-\frac{g(z_{0}r)}{g(z_{0})}\right)\frac{1}{1-r}.
\end{align*}
Taking real parts and using \eqref{ch3,eq3.3}, it follows that
\begin{align*}
    m&=\lim_{r\rightarrow 1} \left(1-\RE
    \frac{g(z_{0}r)}{g(z_{0})}\right)\frac{1}{1-r}\\
     &\geq \lim_{r\rightarrow 1}
     \left(1-\frac{r^{n}(r+|s_{n}|)}{1+|s_{n}|r}\right)\frac{1}{1-r}\\
     &=n+\frac{1-|s_{n}|}{1+|s_{n}|}\\
     &=n+\frac{|g(z_{0})|-|g_{n}|r_{0}^{n}}{|g(z_{0})|+|g_{n}|r_{0}^{n}}.\qedhere
     \end{align*}
%The function $H: (0,1) \rightarrow \mathbb{R}$ given by
%\[H(r)=\frac{a-br^n}{a+br^n}, \quad a\geq 0, b\geq 0  \]
%is a decreasing function of $r$, and so
%\[ m \geq n+\frac{|g(z_{0})|-|g_{n}|}{|g(z_{0})|+|g_{n}|}.\qedhere \]
\end{proof}

%\begin{definition}\cite[Definition 2.2b, p.\ 21]{monograph}
To state the second technical lemma, let $Q$ be the class of
functions $q$ that are analytic and injective in
$\overline{\mathbb{D}}\setminus E(q)$, where \[ E(q) := \left\{
\zeta \in
\partial \mathbb{D}: \lim_{z \rightarrow \zeta}
q(z)=\infty\right\}\] and are such that $q'(\zeta) \neq 0 $ for
$\zeta\in
\partial \mathbb{D}\setminus E(q)$.
%\end{definition}

\begin{lemma}\label{ch3,lem3.2}
Let $q \in Q$ with $q(0)=a$, and $p \in \mathcal{H}_{\beta}[a,n]$
with $p(z)\not\equiv a$. If there exists a point $z_{0} \in
\mathbb{D}$ such that $p(z_{0}) \in q(\partial \mathbb{D})$ and
$p(\{z:|z|<|z_{0}|\})\subset q(\mathbb{D})$, then
\begin{equation}\label{ch3,eq3.4}
z_{0}p'(z_{0})=m\zeta_{0} q'(\zeta_{0})
\end{equation}
and
\begin{equation}\label{ch3,eq3.5}
\RE \left(1+\frac{z_{0}p''(z_{0})}{p'(z_{0})}\right) \geq m\RE
\left(1+\frac{\zeta_{0}q''(\zeta_{0})}{q'(\zeta_{0})}\right),
\end{equation}
where $q^{-1}(p(z_{0}))=\zeta_{0}=e^{i\theta_{0}}$ and
\begin{equation}\label{ch3,eq3.6}
m \geq n+\frac{|q'(0)|-\beta |z_{0}|^{n}}{|q'(0)|+\beta
|z_{0}|^{n}}.
\end{equation}
\end{lemma}

\begin{proof}Except for \eqref{ch3,eq3.6}, the assertions here
follow from Lemma 2.2d in \cite[p.\ 24]{monograph}. Thus, we shall
only prove \eqref{ch3,eq3.6}. Let  $g$ be defined  by
\[g(z)=q^{-1}(p(z)),\quad |z| \leq |z_{0}|.\]
Then $g$ is analytic in $\{z:|z| \leq |z_{0}|\}$ and satisfies
$|g(z_{0})|=1$, $g(0)=0$, $|g(z)| \leq 1$ for $|z| \leq |z_{0}|$ and
\[g^{(k)}(0)=p^{(k)}(0)=0 \quad \text{for} \quad k=1,2,\cdots, n-1.\]
If $g(z)=g_{n}z^{n}+g_{n+1}z^{n+1}+\cdots$, then the relation
$q(g(z))=p(z)$ gives $g_{n}=\beta/q'(0)$. Lemma \ref{ch3,lem3.1} and
Lemma 2.2d in \cite[p.\ 24]{monograph} yield an $m$ satisfying
\eqref{ch3,eq3.4} and \eqref{ch3,eq3.5} where
\[m \geq n+\frac{|g(z_{0})|-|g_{n}||z_{0}|^{n}}{|g(z_{0})|+|g_{n}||z_{0}|^{n}}.\]
Since $|g(z_{0})|=1$, it follows that
\[m \geq n+\frac{1-|g_{n}||z_{0}|^{n}}{1+|g_{n}||z_{0}|^{n}}=
n+\frac{|q'(0)|-\beta |z_{0}|^{n}}{|q'(0)|+\beta |z_{0}|^{n}}.\qedhere\]
\end{proof}

%
%If $p \in \mathcal{H}_{\beta}[a,n]$ satisfies \eqref{ch3,eq2.4} and
%$q$ is an $(a,n)-\beta$-dominant of \eqref{ch3,eq2.4}, we show that
%$|\beta| \leq |q'(0)|$. This follows from $p \prec q$ as shown
%below. By definition of subordination, there exists an analytic
%function $w$ in $\mathbb{D}$ with $w(0)=0$ and $|w(z)|<1$ such that
%\[p(z)=q(w(z)),\quad z \in \mathbb{D}.\]
%Since
%\[p(z)=a+\beta z^{n}+p_{n+1}z^{n+1}+\cdots\quad
%\text{and} \quad q(z)=a+q'(0)z+\cdots\] we see that
%\begin{align*}
%w(z)&=q^{-1}(p(z)) =\frac{\beta}{q'(0)}z^{n}+\cdots.\end{align*} By
%generalized Schwarz Lemma, it follows that $|\beta| \leq |q'(0)|$.

Following Miller and Mocanu \cite{monograph}, two important
functions of $q\in Q$ will be considered, namely, when
$q(\mathbb{D})$ is a disk or a half-plane.\\

\noindent \textbf{Case 1.} The set $q(\mathbb{D})$ is the disk
$\Delta=\mathbb{D}_{M}=\{w:|w|<M\}$. Here the function
\[q(z)=M\frac{Mz+a}{M+\overline{a}z}\quad (z \in \mathbb{D}),\]
with $M>0$ and $|a|<M$, is univalent in $\overline{\mathbb{D}}$ and
satisfies $q(\mathbb{D})=\Delta$, $q(0)=a$ and $q \in Q$.\\

\noindent \textbf{Case 2.} The set $q(\mathbb{D})$ is the half-plane
$\Delta=\{w:\RE w>\alpha\}$. Then
\[q(z)=\frac{a-(2\alpha-\overline{a})z}{1-z}\quad (z \in \mathbb{D}),\]
where $\alpha \in \mathbb{R}$ and $\RE a >\alpha$, is univalent in
$\overline{\mathbb{D}}\setminus \{1\}$ and satisfies
$q(\mathbb{D})=\Delta$, $q(0)=a$ and $q \in Q$.

The following two lemmas are important in the above cases.

\begin{lemma}\label{ch3,lem3.3}
Let $p \in \mathcal{H}_{\beta}[a,n]$, $\beta\neq 0$. If $z_{0} \in
\mathbb{D}$ and
\[|p(z_{0})|=\max_{|z| \leq |z_{0}|} |p(z)|,\]
then
\[\frac{z_{0}p'(z_{0})}{p(z_{0})} \geq
\left(n+\frac{|p(z_{0})|^{2}-|a|^{2}-\beta
|p(z_{0})||z_{0}|^{n}}{|p(z_{0})|^{2}-|a|^{2}+\beta
|p(z_{0})||z_{0}|^{n}}\right)
 \frac{|p(z_{0})-a|^{2}}{|p(z_{0})|^{2}-|a|^{2}}\]
and
\[\RE \left( \frac{z_{0}p''(z_{0})}{p'(z_{0})}+1\right)
\geq \left(n+\frac{|p(z_{0})|^{2}-|a|^{2}-\beta
|p(z_{0})||z_{0}|^{n}}{|p(z_{0})|^{2}-|a|^{2} +\beta
|p(z_{0})||z_{0}|^{n}}\right)\frac{|p(z_{0})-a|^{2}}{|p(z_{0})|^{2}-|a|^{2}}.\]
\end{lemma}

\begin{proof}
Set $M=|p(z_{0})|$. Similar to  the proof of Lemma 2.2e in \cite[p.\
25]{monograph}, Lemma \ref{ch3,lem3.2} shows there exists an $m$
satisfying \eqref{ch3,eq3.6} such that
\begin{equation}\label{ch3,eq3.7}
z_{0}p'(z_{0})=mp(z_{0})\frac{|p(z_{0})-a|^{2}}{|p(z_{0})|^{2}-|a|^{2}}
\quad \mbox{and} \quad \RE
\left(\frac{z_{0}p''(z_{0})}{p'(z_{0})}+1\right) \geq m
\frac{|p(z_{0})-a|^{2}}{|p(z_{0})|^{2}-|a|^{2}}.
\end{equation}
Since
\[q'(z)=M\frac{M^{2}-|a|^{2}}{(M+\overline{a}z)^{2}},\]
it readily follows that
\begin{align*}
    q'(0)&=\frac{|p(z_{0})|^{2}-|a|^{2}}{|p(z_{0})|}.
\end{align*}
Therefore \eqref{ch3,eq3.6} becomes
\[m \geq n+\frac{|p(z_{0})|^{2}-|a|^{2}-\beta |p(z_{0})||z_{0}|^{n}}{|p(z_{0})|^{2}-|a|^{2}+\beta |p(z_{0})||z_{0}|^{n}},\]
so that \eqref{ch3,eq3.7} gives the desired result.
\end{proof}

\begin{remark}\label{ch3,rem3.4}
If $a=0$, then  Lemma \ref{ch3,lem3.3} reduces to Lemma
\ref{ch3,lem3.1}.
\end{remark}

\begin{lemma}\label{ch3,lem3.5}
Let $p \in \mathcal{H}_{\beta}[a,n]$, $\beta\neq0$. If $z_{0} \in
\mathbb{D}$ and
\[\RE p(z_{0})=\min_{|z| \leq |z_{0}|} \RE p(z),\]
then
\[z_{0}p'(z_{0}) \leq -\frac{1}{2}\left(n+\frac{2
|\RE (a-p(z_{0}))|-\beta |z_{0}|^{n}}{2 |\RE (a-p(z_{0}))|+\beta
|z_{0}|^{n}}\right) \frac{|p(z_{0})-a|^{2}}{\RE (a-p(z_{0}))}\] and
\[\RE \left( \frac{z_{0}p''(z_{0})}{p'(z_{0})}+1\right) \geq 0.\]
\end{lemma}

\begin{proof}
If we set $\alpha=\RE p(z_{0})$, then Lemma 2.2f in  \cite[p.\
26]{monograph} and  Lemma  \ref{ch3,lem3.2} give the existence of
$m$ satisfying \eqref{ch3,eq3.6} such that
\begin{equation}\label{ch3,eq3.8}
z_{0}p'(z_{0})=-\frac{m}{2}\frac{|p(z_{0})-a|^{2}}{\RE[a-p(z_{0})]}
\quad \mbox{and} \quad \RE
\left(\frac{z_{0}p''(z_{0})}{p'(z_{0})}+1\right) \geq 0.
\end{equation}
Now
\[q'(z)=\frac{2(\RE a-\alpha)}{(1-z)^{2}}\]
implies
\begin{align*}
    q'(0)&=2(\RE a-\alpha)=2\RE (a- p(z_{0})).
\end{align*}
Therefore \eqref{ch3,eq3.6} becomes
\[m \geq n+\frac{2 |\RE (a-p(z_{0}))|-\beta|z_{0}|^{n} }{2 |\RE (a-p(z_{0}))|+\beta |z_{0}|^{n}},\]
so that \eqref{ch3,eq3.8} gives the desired result.
\end{proof}

\section{$\beta$-admissible functions and fundamental theorems}\label{ch3,sec4}
\noindent In this section we will prove fundamental results related
to the implication
\[\{\psi(p(z),z p'(z),z^{2}p''(z)): z \in \mathbb{D}\}
\subset \Omega \quad \Rightarrow \quad  p(\mathbb{D}) \subset
\Delta\] for a suitably defined class of functions.

\begin{definition}[$\beta$-Admissibility Condition]\label{ch3,def4.1}
Let $\Omega$ be a domain in $\mathbb{C}$, $q \in {Q}$, and $\beta
\in \mathbb{C}$ with $\beta \leq |q'(0)|$. The class
$\Psi_{n,\beta}(\Omega,q)$ consists of
\boldmath\textbf{$\beta$-admissible functions}\unboldmath
$\psi:\mathbb{C}^3\rightarrow \mathbb{C}$   satisfying the following
conditions:
\begin{itemize}
  \item[(i)] $ \psi(r,s,t)$ is continuous in a domain $D\subset \mathbb{C}^3$,
  \item[(ii)] $(q(0),0,0)\in{D}$ and $\psi(q(0),0,0)\in{\Omega}$,
  \item[(iii)] $\psi(r_0,s_0,t_0)\not\in{\Omega}$ whenever $(r_0,s_0,t_0)\in{D}$, $r_0=q(\zeta)$, $s_0=m \zeta  q'(\zeta)$ and
  \[ \RE \left( \frac{t_0}{s_0}+1\right) \geq m \RE \left( \frac{\zeta q''(\zeta)}{q'(\zeta)}+1\right),\]
  where $|\zeta|=1$, $q(\zeta)$ is finite and
  \[m \geq n+\frac{|q'(0)|-\beta }{|q'(0)|+\beta }.\]
\end{itemize}
The class $\Psi_{1,\beta}(\Omega,q)$ is denoted by
$\Psi_{\beta}(\Omega,q)$.
\end{definition}

Note that $\Omega$ is not required to be simply-connected or has a particularly nice boundary as we do for $q(\mathbb{D})$. If $\beta=|q'(0)|$,
then the concept of $\beta$-admissibility coincides with the usual admissibility as discussed in \cite{millermocanu1981}, that is,
$\Psi_n\equiv\Psi_{n,|q'(0)|}$.  It is also evident from the definition that
\[  \Psi_n\equiv\Psi_{n,|q'(0)|} \subseteq \Psi_{n,\beta_1} \subseteq
\Psi_{n,\beta_2}\subseteq \Psi_{n,0}\equiv\Psi_{n+1}  \quad ( 0\leq
\beta_2 \leq \beta_1 \leq |q'(0)| ).
\] In view of the above inclusions, it is assumed throughout this sequel that $0<\beta  \leq |q'(0)|$.

\begin{theorem}\label{ch3,th4.2}
Let $q(0)=a$, $\psi \in \Psi_{n,\beta}(\Omega,q)$ with associated
domain D, and $\beta \in \mathbb{C}$ with $0<\beta  \leq |q'(0)|$.
Let $p \in \mathcal{H}_{\beta}[a,n]$. If $(p(z), zp'(z),
z^{2}p''(z))\in {D}$ and
\begin{equation}\label{ch3,eq4.1}
 \psi (p(z), zp'(z), z^{2}p''(z)) \in {\Omega}\quad ( z\in
 {\mathbb{D}}),
\end{equation}
 then $p\prec q$.
\end{theorem}

\begin{proof}
Taking into account that  $q$ is univalent in $\mathbb{D}$ and
$p(0)=q(0)=a$, it remains to show that $p(\mathbb{D}) \subset
q(\mathbb{D})$. Let, if possible, $p(\mathbb{D}) \not\subset
q(\mathbb{D})$. Then there exists a point $z_{0} \in \mathbb{D}$ for
which $p(\{z:|z|<|z_{0}|\}) \subset q(\mathbb{D})$ and $p(\{z:|z|
\leq |z_{0}|\}) \not\subset q(\mathbb{D})$. Since $p(\{z:|z| \leq
|z_{0}|\}) \subset q(\overline{\mathbb{D}})$, therefore $p(z_{0})
\in \partial q(\mathbb{D})=q(\partial \mathbb{D})$. At the point
$z_{0}$,  Lemma \ref{ch3,lem3.2} shows that
\[p(z_{0})=q(\zeta_{0}),\quad z_{0}p'(z_{0})=m\zeta_{0}q'(z_{0}),\]
and
\[\RE \left(\frac{z_{0}p''(z_{0})}{p'(z_{0})}+1\right) \geq
m\RE \left(\frac{\zeta_{0}q''(\zeta_{0})}{q'(\zeta_{0})}+1\right),\]
where $|\zeta_{0}|=1$, $q(\zeta_{0})$ is finite and
\[m \geq n+\frac{|q'(0)|-\beta |z_{0}|^{n}}{|q'(0)|+\beta |z_{0}|^{n}}.\]
The function $H: [0,1] \rightarrow \mathbb{R}$ given by
\[H(r)=\frac{a-br^n}{a+br^n}, \quad a\geq 0, b\geq 0  \]
is a decreasing function of $r$, and so
\[ m \geq n+\frac{|q'(0)|-\beta}{|q'(0)|+\beta}.\]
With $r_{0}=p(z_{0})$, $s_{0}=z_{0}p'(z_{0})$ and
$t_{0}=z_{0}^{2}p''(z_{0})$ in   part $(iii)$ of Definition
\ref{ch3,def4.1}, the condition
\[\psi (p(z_{0}), z_{0}p'(z_{0}), z_{0}^{2}p''(z)) \not\in {\Omega}\]
 contradicts \eqref{ch3,eq4.1}. Hence $p(\mathbb{D}) \subset
q(\mathbb{D})$ and $p \prec q$.
\end{proof}

The proof of the following result is similar to Corollary 1.1 in
\cite[p.\ 160]{millermocanu1981}.
\begin{corollary}\label{ch3,cor4.3}
Let $q$ be univalent in $\mathbb{D}$ with $q(0)=a$, and
$q_{\rho}(z)=q(\rho z)$, $0< \rho <1$. Let $\psi \in
\Psi_{n,\beta}(\Omega,q_{\rho})$ with domain $D$, $0< \rho <1$,
where $\beta \in \mathbb{C}$ with $0<\beta  \leq |q'(0)|$, and  $p
\in \mathcal{H}_{\beta}[a,n]$. If $(p(z), zp'(z), z^{2}p''(z))\in
{D}$ and
\[ \psi (p(z), zp'(z), z^{2}p''(z)) \in {\Omega}\quad ( z\in
 {\mathbb{D}}),\]  then $p\prec q$.
\end{corollary}

\begin{definition}\label{ch3,def4.4}
Let $\Omega \neq \mathbb{C}$ be a simply connected domain in
$\mathbb{C}$, $q \in {Q}$ and $\beta \in \mathbb{C}$. Let $h$ be a
conformal mapping of $\mathbb{D}$ onto $\Omega$. Denote by
$\Psi_{n,\beta}(h,q)$ the class of functions $\psi \in
\Psi_{n,\beta}(\Omega,q)=\Psi_{n,\beta}(h(\mathbb{D}),q)$ which are
analytic in their associated domains $D$ and satisfy
$\psi(q(0),0,0)=h(0)$. We write $\Psi_{1,\beta}(h,q)$ as
$\Psi_{\beta}(h,q)$.
\end{definition}

The following theorem and corollary are immediate consequences of
Theorem \ref{ch3,th4.2} and Corollary \ref{ch3,cor4.3}.
\begin{theorem}\label{ch3,th4.5}
Let $q(0)=a$ and  $\psi \in \Psi_{n,\beta}(h,q)$ with associated
domain D, where $\beta \in \mathbb{C}$ with $0<\beta  \leq |q'(0)|$.
Let $p \in \mathcal{H}_{\beta}[a,n]$. If $(p(z), zp'(z),
z^{2}p''(z))\in {D}$ and
\[ \psi (p(z), zp'(z), z^{2}p''(z)) \prec h(z) \quad ( z\in
 {\mathbb{D}}),\]
  then $p\prec q$.
\end{theorem}

\begin{corollary}\label{ch3,cor4.6}
Let $h$ and $q$ be univalent in $\mathbb{D}$ with $q(0)=a$, and let
$h_{\rho}(z)=h(\rho z)$, $q_{\rho}(z)=q(\rho z)$, for $0<\rho<1$.
Let $\psi \in \Psi_{n,\beta}(h_{\rho},q_{\rho})$ with domain $D$,
for $0<\rho<1$, where $\beta \in \mathbb{C}$ with $0<\beta  \leq
|q'(0)|$ and let $p \in \mathcal{H}_{\beta}[a,n]$. If $(p(z),
zp'(z), z^{2}p''(z))\in {D}$  and
\[ \psi (p(z), zp'(z), z^{2}p''(z)) \prec h(z)\quad ( z\in
 {\mathbb{D}}),\]  then $p\prec q$.
\end{corollary}

\subsection{Two special cases}
\noindent Let us next formulate the theorems above to the two
important examples of $q(\mathbb{D})$ being a disk and
$q(\mathbb{D})$ being a half-plane considered earlier.\\

\noindent \textbf{Case 1.} The disk
$\Delta=\mathbb{D}_{M}=\{w:|w|<M\}$. Here the function
\[q(z)=M\frac{Mz+a}{M+\overline{a}z}\quad (z\in \mathbb{D}),\]
where $M>0$ and $|a|<M$, is univalent in $\overline{\mathbb{D}}$ and
satisfies $q(\mathbb{D})=\Delta$, $q(0)=a$ and $q \in Q$. To
describe the class of $\beta$-admissible functions in this case, set
$\Psi_{n,\beta}(\Omega,M,a):=\Psi_{n,\beta}(\Omega,q)$ and when
$\Omega=\Delta$, denote the class by $\Psi_{n,\beta}(M,a)$.
Comparing with Lemma 2.2e in \cite[p.\ 25]{monograph}, the condition
of $\beta$-admissibility becomes
\begin{equation}\label{ch3,eq4.2}
\begin{split}
\psi(r,s,t) &\not\in \Omega \quad \mbox{whenever} \quad (r,s,t) \in D,\\
r&=q(\zeta)=Me^{i\theta},\\
s&=m\zeta q'(\zeta)=m\frac{M|M-\overline{a}e^{i \theta}|^{2}}{M^{2}-|a|^{2}}e^{i \theta}, \quad \mbox{and}\\
\RE \left(\frac{t}{s}+1\right) &\geq m \frac{|M-\overline{a}e^{i
\theta}|^{2}}{M^{2}-|a|^{2}},
\end{split}
\end{equation}
where $\theta \in \mathbb{R}$ and
\[m \geq n+\frac{M^{2}-|a|^{2}-M\beta }{M^{2}-|a|^{2}+M\beta }.\]
Thus, the class $\Psi_{n,\beta}(\Omega,M,a)$ consists of those
functions $\psi:\mathbb{C}^3\rightarrow \mathbb{C}$ that are
continuous in a domain $D \subset \mathbb{C}^3$ with $(a,0,0) \in D$
and $\psi(a,0,0) \in \Omega$, and  satisfying the
$\beta$-admissibility condition \eqref{ch3,eq4.2}.

If $a=0$, then \eqref{ch3,eq4.2} simplifies to
\[\psi(r,s,t) \not\in \Omega \quad \mbox{whenever} \quad (r,s,t) \in D,\
r=Me^{i\theta},\  s=mMe^{i \theta},\  \mbox{and}\]
\[  \RE \left(\frac{t}{s}+1\right) \geq m, \]
where $\theta \in \mathbb{R}$ and
\[m \geq n+\frac{M-\beta }{M+\beta }.\]
Equivalently, the condition is
\begin{equation}\label{ch3,eq4.3}
\begin{split}
\psi(Me^{i \theta},Ke^{i \theta},L) &\not\in \Omega \quad
\mbox{whenever}
\quad (Me^{i \theta},Ke^{i \theta},L) \in D,\\
K &\geq \left(n+\frac{M-\beta }{M+\beta }\right)M,\quad
\mbox{and}\quad
 \RE (Le^{-i \theta})  \geq \left(n-\frac{2\beta }{M+\beta }\right)K,
\end{split}
\end{equation}
where $\theta \in \mathbb{R} $ and $n \geq 1$.

In this particular case, Theorem \ref{ch3,th4.2} can be expressed in
the following form:
\begin{theorem}\label{ch3,th4.7}
 Let $p \in \mathcal{H}_{\beta}[a,n]$ with
$|a|<M$, $0<\beta  \leq (M^{2}-|a|^{2})/M$, $M>0$.
\begin{enumerate}
  \item [(i)]
Let $\psi \in \Psi_{n,\beta}(\Omega,M,a)$ with associated domain
$D$. If $(p(z), zp'(z), z^{2}p''(z))\in {D}$ and
\[ \psi (p(z), zp'(z), z^{2}p''(z)) \in \Omega\quad ( z\in
 {\mathbb{D}}),  \]
then $|p(z)|<M.$
  \item [(ii)] Let $\psi \in \Psi_{n,\beta}(M,a)$ with associated domain $D$.
   If $(p(z), zp'(z), z^{2}p''(z))\in {D}$  and
\[| \psi (p(z), zp'(z), z^{2}p''(z))|<M \quad ( z\in
 {\mathbb{D}}) ,\]
 then
 $|p(z)|<M.$\\
\end{enumerate}
\end{theorem}
%
%Examples of these results will be given in the next section.\\
\noindent \textbf{Case 2.}  The half-plane $\Delta=\{w:\RE w>0\}$.
Here the function
\[q(z)=\frac{a+\overline{a}z}{1-z}\quad (z\in \mathbb{D}),\]
where $\RE a >0$, is univalent in
$\overline{\mathbb{D}}\setminus\{1\}$ and satisfies
$q(\mathbb{D})=\Delta$, $q(0)=a$ and $q \in Q$. Let
$\Psi_{n,\beta}(\Omega,a):=\Psi_{n,\beta}(\Omega,q)$ and when
$\Omega=\Delta$, denote the class by $\Psi_{n,\beta}(a)$. Comparing
with  Lemma 2.2f in \cite[p.\ 26]{monograph}, the condition of
$\beta$-admissibility becomes
\[\psi(r,s,t) \not\in \Omega \quad \mbox{whenever} \quad (r,s,t) \in D,\]
\begin{align*}
 r&=q(\zeta)=i\rho,\\
  s&=m\zeta q'(\zeta)=-\frac{m}{2}\frac{|a-i\rho|^{2}}{\RE a}, \quad \mbox{and}\\
  \RE \left(\frac{t}{s}+1\right) &\geq 0,
\end{align*}
where $\rho \in \mathbb{R}$ and
\[m \geq n+\frac{2\RE a-\beta }{2\RE a+\beta }.\]
Equivalently,
\begin{equation}\label{ch3,eq4.4}
\begin{split}
\psi(i\rho,\sigma,\mu+i\nu) &\not\in \Omega \quad \mbox{whenever} \quad (i\rho,\sigma,\mu+i\nu) \in D,\\
  \sigma &\leq-\frac{1}{2}\left( n+\frac{2\RE a-\beta }{2\RE a+\beta }\right)\frac{|a-i\rho|^{2}}{\RE a},
   \quad \mbox{and}\\
 \sigma+\mu &\leq 0,
\end{split}
\end{equation}
where $\rho,\sigma,\mu,\nu \in \mathbb{R}$ and $n \geq 1$. Thus, the
class $\Psi_{n,\beta}(\Omega,a)$ consists of those functions
$\psi:\mathbb{C}^3\rightarrow \mathbb{C}$ that are continuous in a
domain $D \subset \mathbb{C}^3$ with $(a,0,0) \in D$ and
$\psi(a,0,0) \in \Omega$,   satisfying the $\beta$-admissibility
condition \eqref{ch3,eq4.4}.

If $a=1$, then \eqref{ch3,eq4.4} simplifies to
\begin{equation}\label{ch3,eq4.5}
\begin{split}
\psi(i\rho,\sigma,\mu+i\nu) &\not\in \Omega \quad \mbox{whenever} \quad (i\rho,\sigma,\mu+i\nu) \in D,\\
                                \sigma &\leq-\frac{1}{2}\left( n+\frac{2-\beta }{2+\beta }\right)
                                 (1+\rho^{2}), \quad \mbox{and}\\
                             \sigma+\mu &\leq 0,
\end{split}
\end{equation}
where $\rho,\sigma,\mu,\nu \in \mathbb{R}$, and $n \geq 1$, a
condition much easier to check.

In this particular case, Theorem \ref{ch3,th4.2} can be rephrased in
the following form.
\begin{theorem}\label{ch3,th4.8}
Let $p \in \mathcal{H}_{\beta}[a,n]$ with $\RE a>0$, $0<\beta  \leq
2\RE a$.
\begin{enumerate}
  \item [(i)] Let $\psi \in \Psi_{n,\beta}(\Omega,a)$ with associated domain $D$. If $(p(z), zp'(z), z^{2}p''(z))\in {D}$  and
                 \[ \psi (p(z), zp'(z), z^{2}p''(z)) \in \Omega\quad ( z\in
 {\mathbb{D}}),  \]
             then $\RE p(z) >0.$
  \item [(ii)] Let $\psi \in \Psi_{n,\beta}(a)$ with associated domain $D$.
   If $(p(z), zp'(z), z^{2}p''(z))\in {D}$ and
                 \[\RE \psi (p(z), zp'(z), z^{2}p''(z))>0 \quad ( z\in
 {\mathbb{D}}),\]
              then
                 $\RE p(z)>0.$
\end{enumerate}
\end{theorem}

\section{Examples}\label{ch3,sec5}\noindent
In this section,  examples of differential inequalities and subordinations are presented  to obtain several interesting results. These are
applications of $\beta$-admissible functions $\psi$ in $\Psi_{n,\beta}(\Omega,q)$, by  judicious choices  of $\psi$. For the sake of comparison,
we shall look at several examples that were considered by Miller and Mocanu in \cite[ pp.\ 36--42]{monograph}.

The first example involves a disk of radius $M$ and is an
application of Theorem~\ref{ch3,th4.7}.

\begin{example}\label{ch3,ex5.1}
Let $\psi(r,s,t)=r+s+t$, $a=0$, and $\Omega=h(\mathbb{D})$, where
\[h(z)=\left(\left(n+\frac{M-\beta }{M+\beta }\right)^{2}+1\right)Mz.\]To apply Theorem \ref{ch3,th4.7}, we need to show that $\psi \in
\Psi_{n,\beta}(\Omega,M,0)$ for $n \geq 1$ and $\beta  \leq M$. The
function $\psi$ evidently satisfies the first two admissibility
conditions: $\psi$ is continuous in the domain $D=\mathbb{C}^{3}$,
$(0,0,0) \in D$ and $\psi(0,0,0)=0 \in \Omega$. It remains to show
the $\beta$-admissibility condition \eqref{ch3,eq4.3} is satisfied.
Since
\[\psi(Me^{i\theta},Ke^{i\theta},L)=Me^{i\theta}+Ke^{i\theta}+L,\]
then
\begin{align*}
    |\psi(Me^{i\theta},Ke^{i\theta},L)|&=|M+K+Le^{-i\theta}|\\
&\geq M+K+\RE(Le^{-i\theta})\\
&\geq M+K+\left(n-\frac{2\beta }{M+\beta }\right)K\\
&=M+\left(n+\frac{M-\beta }{M+\beta }\right)K\\
& \geq M+\left(n+\frac{M-\beta }{M+\beta }\right)^{2}M\\
&= \left(\left(n+\frac{M-\beta }{M+\beta }\right)^{2}+1\right)M,
\end{align*}
whenever $K \geq (n+(M-\beta )/(M+\beta ))M$, $\RE (Le^{-i\theta})
\geq \big(n-2\beta /(M+\beta )\big)K$, $\theta \in \mathbb{R}$, and
$n \geq 1$. Thus, $\psi \in \Psi_{n,\beta}(\Omega,M,0)$ for $n \geq
1$.
Theorem \ref{ch3,th4.7}  now yields the following differential subordination result:\\

\emph{Let $p \in \mathcal{H}_{\beta}[0,n]$ with $0<\beta  \leq M$.
If
\[|p(z)+zp'(z)+z^{2}p''(z)|
<\left(\left(n+\frac{M-\beta }{M+\beta }\right)^{2}+1\right)M
\quad(z \in \mathbb{D}),\] then $|p(z)|<M.$}
\end{example}

\begin{remark}\label{ch3,rem5.2}
If $\beta =M$, then Example  \ref{ch3,ex5.1} reduces to Example 2.4a
in \cite[p.\ 36]{monograph}. Since \[(n^{2}+1)M \leq
\left(\left(n+\frac{M-\beta }{M+\beta }\right)^{2}+1\right)M \quad
\text{ for\ } \beta \leq M,\] it is clear that Example
\ref{ch3,ex5.1} extends  Example 2.4a in \cite[p.\ 36]{monograph}
for functions $p \in \mathcal{H}_{\beta}[0,n]$.
\end{remark}

The next example involves a function  with positive real part and is
an application of Theorem \ref{ch3,th4.8}.
\begin{example}\label{ch3,ex5.3}
Let $\psi(r,s,t)=1-r^{2}+5s+t$, $a=1$ and
\[\Omega=\left\{w:|w|<\frac{6-\beta }{2+\beta }\right\}.\]
Now Theorem \ref{ch3,th4.8} is applicable provided $\psi \in \Psi_{n,\beta}(\Omega,1)$ for $n \geq 1$ and $\beta  \leq 2$. The function $\psi$
 is continuous in the domain $D=\mathbb{C}^{3}$, $(1,0,0) \in D$ and $\psi(1,0,0) \in \Omega$.
To show that the $\beta$-admissibility condition \eqref{ch3,eq4.5}
is satisfied, consider
  $\psi(i\rho,\sigma,\mu+i\nu)=1+\rho^{2}+5\sigma+\mu+i\nu$.
Then
  \begin{align*}
   |\psi(i\rho,\sigma,\mu+i\nu)| &=|1+\rho^{2}+5\sigma+\mu+i\nu|\\
                                 &\geq -(1+\rho^{2}+5\sigma+\mu)\\
                         %        &= -(1+\rho^{2})-4\sigma-(\sigma+\mu)\\
                                 & \geq -(1+\rho^{2})-4\sigma\\
                                 & \geq -(1+\rho^{2})+2\left(n+\frac{2-\beta }{2+\beta }\right)(1+\rho^{2})\\
                                 & \geq 2\left(n+\frac{2-\beta }{2+\beta }\right)-1\\
                                 & \geq 2\left(1+\frac{2-\beta }{2+\beta }\right)-1\\
                                 &=\frac{6-\beta }{2+\beta }
    \end{align*}
whenever $\rho \in \mathbb{R}$, $\sigma \leq -\left(n+(2-\beta
)/(2+|\beta)\right)(1+\rho^{2})/2$, $\sigma+\mu \leq 0$ and $n \geq
1$. Thus,
$\psi \in \Psi_{n,\beta}(\Omega,1)$, and Theorem \ref{ch3,th4.8} yields the following differential inequality result:\\

\emph{Let $p \in \mathcal{H}_{\beta}[1,n]$, and $0<\beta  \leq 2$.
If
\[|z^{2}p''(z)+5zp'(z)-p^{2}(z)+1|<\frac{6-\beta }{2+\beta } \quad
  \quad ( z \in \mathbb{D}),\]
then $\RE p(z) >0.$}
\end{example}

\begin{remark}\label{ch3,rem5.4}
If $\beta =2$, then Example \ref{ch3,ex5.3} reduces to Example 2.4i
in \cite[p.\ 41]{monograph}.  Since
\[\frac{6-\beta }{2+\beta }\geq 1 \quad  \text{\ for }\ \beta \leq2,\]
Example \ref{ch3,ex5.3} extends  Example 2.4i in \cite[p.\
41]{monograph}.
\end{remark}

The next two examples illustrate the sensitivity of the class of
$\beta$-admissible functions to the value $n$.
\begin{example}\label{ch3,ex5.5}
Let $\psi(r,s,t)=r+s+1-r^{2}$, $\Omega=\Delta=\{w:\RE w>0\}$ and
$a=1$. We first show that $\psi \in \Psi_{n,\beta}(1)$ for $n \geq
(2+3\beta )/(2+\beta )$ and $\beta  \leq 2$. The function
  $\psi$ is continuous in the domain $D=\mathbb{C}^{3}$,
  $(1,0,0) \in D$, and $\RE \psi(1,0,0)=1>0$ so that $\psi(1,0,0) \in \Omega.$
To verify  the $\beta$-admissibility condition \eqref{ch3,eq4.5} is
satisfied, consider

  \[\psi(i\rho, \sigma,\mu+i\nu)=i\rho+\sigma+1+\rho^{2}.\]
Then
  \begin{align*}
    \RE \psi(i\rho, \sigma,\mu+i\nu)&=\sigma+1+\rho^{2}\\
                                    &\leq -\frac{1}{2}\left( n+\frac{2-\beta }{2+\beta }\right)(1+\rho^{2})+1+\rho^{2}\\
                                    &=(1+\rho^{2})\left[-\frac{n}{2}+\frac{2+3\beta }{2(2+\beta )}\right]\\
                                    &=\frac{1}{2}(1+\rho^{2})\left[\frac{2+3\beta }{2+\beta }-n\right] \leq 0,
  \end{align*}
whenever $\rho \in \mathbb{R}$, $\sigma \leq -[ n+(2-\beta )/(2+\beta )](1+\rho^{2})/2$, and $n \geq (2+3\beta )/(2+\beta )$.\\
Thus, $\psi \in \Psi_{n,\beta}(1)$ for $n  \geq (2+3\beta )/(2+\beta
)$ and $\beta  \leq 2$. Therefore, by Theorem \ref{ch3,th4.8}, the
following differential inequality
result is obtained:\\

\emph{Let $p \in \mathcal{H}_{\beta}[1,n]$ with  $0<\beta  \leq 2$
and $n \geq (2+3\beta )/(2+\beta )$. If
\[\RE \left(p(z)+zp'(z)+1-p^{2}(z)\right) >0 \quad ( z\in
 {\mathbb{D}}),\]
then $\RE p(z) >0.$}
\end{example}

\begin{remark}\label{ch3,rem5.6}
If $\beta =2$, then \[\frac{2+3\beta }{2+\beta }=2\] so that Example
\ref{ch3,ex5.5} reduces to Example~2.4g in \cite[ p.\
40]{monograph}. For $\beta  \leq 2$, then \[\frac{2+3\beta }{2+\beta
} \leq 2,\] and  Example \ref{ch3,ex5.5} extends Example 2.4g in
\cite[ p.\ 40]{monograph}.
\end{remark}

 The next example has a similar restriction.

\begin{example}\label{ch3,ex5.7}
Let $\psi(r,s,t)=2-r^{2}+3s+t$, $\Omega=\Delta=\{w:\RE w>0\}$ and
$a=1$.  The function $\psi$ is continuous in the domain
$D=\mathbb{C}^{3}$,  $(1,0,0) \in D$, and $\RE \psi(1,0,0)=1>0$ so
that
  $\psi(1,0,0) \in \Omega.$
 To show the $\beta$-admissibility condition \eqref{ch3,eq4.5} is satisfied, consider

  \[\psi(i\rho, \sigma,\mu+i\nu)=2+\rho^{2}+3\sigma+\mu+i\nu.\]
Then
  \begin{align*}
    \RE \psi(i\rho, \sigma,\mu+i\nu)&=2+\rho^{2}+2\sigma+(\sigma+\mu)\\
                                    &\leq 2+\rho^{2}+2\sigma\\
                                    &=\left(2-n-\frac{2-\beta }{2+\beta }\right)+\left(1-n-\frac{2-\beta }{2+\beta }\right)\rho^{2}\\
                                    &=\left(\frac{2+3\beta }{2+\beta }-n\right)+\left(\frac{2\beta }{2+\beta }-n\right)\rho^{2} \leq 0,
  \end{align*}
whenever $\rho \in \mathbb{R}$,  $\sigma \leq -\left( n+(2-\beta
)/(2+\beta )\right)(1+\rho^{2})/2$, $\sigma+\mu \leq 0$ and $n \geq
(2+3\beta )/(2+\beta )$.  Thus, $\psi \in \Psi_{n,\beta}(1)$ for $n
\geq (2+3\beta )/(2+\beta )$ and $\beta  \leq 2$. Therefore, Theorem
\ref{ch3,th4.8} yields the following differential inequality result:\\

\emph{Let $p \in \mathcal{H}_{\beta}[1,n]$, $0<\beta  \leq 2$ and $n
\geq (2+3\beta )/(2+\beta )$. If
\[\RE \left(z^{2}p''(z)+3zp'(z)-p^{2}(z)+2\right) >0 \quad ( z\in
 {\mathbb{D}}),\]
then $\RE p(z) >0.$}
\end{example}

\begin{remark}\label{ch3,rem5.8}
If $\beta =2$, then $n \geq (2+3\beta )/(2+\beta )=2$ so that
Example \ref{ch3,ex5.7} extends  Example 2.4h in \cite[p.\
40]{monograph}.
\end{remark}

\begin{example}\label{ch3,ex5.9}
In this example, consider the class $\Psi_{2,\beta}(\Omega,q)$, where $q(z)=1+z$, $\beta  \leq 1$, and
\[\Omega=\left\{w: |w| <\frac{4(3+2\beta )}{(1+\beta )^{2}}\right\}.\]
The function $\psi(r,s,t)=1-r^{2}+3s+t$  is continuous in the domain
$D=\mathbb{C}^{3}$, $(1,0,0) \in D$ and $\psi(1,0,0)\in \Omega$. It
remains to show that the $\beta$-admissibility condition is
satisfied. If we set $r_0=q(\zeta)$, $s_0=m \zeta  q'(\zeta)$ and
  \[ \RE \left( \frac{t_0}{s_0}+1\right) \geq m \RE \left(\zeta \frac{q''(\zeta)}{q'(\zeta)}+1\right),\]
where $|\zeta|=1$ and $m \geq 2+(1-\beta )/(1+\beta )$, then
  \[r_{0}=1+\zeta, \quad s_{0}=m\zeta, \quad \mbox{and}
  \quad \RE(t_{0}\overline{\zeta}) \geq m(m-1).\]
Since
  $\psi(r_{0},s_{0},t_{0})=(3m-2)\zeta-\zeta^{2}+t_{0}$, it follows
  that
  \begin{align*}
    |\psi(r_{0},s_{0},t_{0})|&=|3m-2-\zeta+t_{0}\overline{\zeta}|\\
                            &\geq 3m-2-\RE \zeta+\RE (t_{0}\overline{\zeta})\\
                            & \geq 3m-3+m(m-1)\\
                            &=(m+3)(m-1)\\
                            &\geq \left(5+\frac{1-\beta }{1+\beta }\right)\left(1+\frac{1-\beta }{1+\beta }\right)\\
                            &=\frac{4(3+2\beta )}{(1+\beta )^{2}}.
  \end{align*}
Thus
 $\psi(r_{0},s_{0},t_{0}) \not\in \Omega,$ and the $\beta$-admissibility condition is satisfied, that is,   $\psi \in
\Psi_{2,\beta}(\Omega,q)$. Theorem \ref{ch3,th4.2}  now yields the following :\\

\emph{Let $p \in \mathcal{H}_{\beta}[1,2]$ and $0<\beta  \leq 1$. If
\[|z^{2}p''(z)+3zp'(z)-p^{2}(z)+1| <\frac{4(3+2\beta )}{(1+\beta )^{2}} \quad ( z\in
 {\mathbb{D}}),\]
then
\[|p(z)-1| <1 \quad (z \in \mathbb{D}).\]}
\end{example}

\begin{remark}\label{ch3,rem5.10}
If $\beta =1$, then
\[\frac{4(3+2\beta )}{(1+\beta )^{2}}=5\]
so that Example \ref{ch3,ex5.9} extends Example 2.4k in \cite[p.\
42]{monograph}. The assumption $p \in \mathcal{H}_{\beta}[1,2]$ for
$\beta \leq 1$ implies that
\[\frac{4(3+2\beta )}{(1+\beta )^{2}} \geq 5.\]
\end{remark}

\section{Applications in univalent function theory}\label{ch3,sec6}
\noindent This section looks at several interesting applications of
the theory developed in the earlier sections to normalized convex
and starlike univalent functions with fixed second coefficient. The
results obtained here extend those given in Section 2.6 of  \cite[
p.\ 56]{monograph}.

 Let $\mathcal{A}_n$ be the class consisting of analytic functions $f$ defined in
$\mathbb{D}$ of the form
$f(z)=z+a_{n+1}z^{n+1}+a_{n+2}z^{n+2}+\cdots$, and
$\mathcal{A}:=\mathcal{A}_1$. The class $\mathcal{S}^*(\alpha)$ of
starlike functions of order $\alpha$, $0 \leq \alpha <1$, consists
of functions $f\in \mathcal{A}$ satisfying the inequality
\[ \RE\left(\frac{zf'(z)}{f(z)}\right)>\alpha \quad ( z\in
 {\mathbb{D}}). \] Similarly,  the class
$\mathcal{C}(\alpha)$ of convex functions of order $\alpha$, $0 \leq
\alpha <1$, consists of functions $f\in \mathcal{A}$ satisfying the
inequality
\[ \RE\left(1+\frac{zf''(z)}{f'(z)}\right)>\alpha \quad ( z\in
 {\mathbb{D}}). \]
When $\alpha=0$, these classes are respectively denoted by
$\mathcal{S}^*$ and $\mathcal{C}$. Let $\mathcal{A}_{n,b}$ denote
the class of functions $f \in \A_{n}$ of  the form
\[f(z)=z+bz^{n+1}+a_{n+2}z^{n+2}+\cdots,\]
with fixed coefficient $a_{n+1}=b$. We write $\mathcal{A}_{1,b}$ as
$\mathcal{A}_{b}$.

\begin{theorem}\label{ch3,th6.1}
If $f (z)=z+a_2z^2+\ldots\in \mathcal{C}$, then
$f\in\mathcal{S}^*(\alpha)$ where $\alpha$ is the smallest positive
root of the equation
\begin{equation}\label{ch3,eq6.1}
2 \alpha^{3}-\alpha^{2}|a_2|-4\alpha+2=0
\end{equation}
in the interval $[1/2,2/3]$.
\end{theorem}

\begin{proof}The function $g$ defined by
\[g(\alpha)=2 \alpha^{3}-\alpha^{2}|a_2|-4\alpha+2\]
is continuous in $[1/2,2/3]$. Let $b:=a_2$. Since $f\in\mathcal{C}$,
$|b|\leq1$ and $g$ satisfies
\begin{align*}
    g\left(\frac{1}{2}\right)=\frac{1}{4}(1-|b|) \geq 0,\quad
    \mbox{and}\quad
    g\left(\frac{2}{3}\right)=-\frac{2}{27}(1+6|b|) \leq 0.
\end{align*}
 Therefore  there
exists a root of $g(\alpha)=0$ in $[1/2,2/3]$. Define the function
$p:\mathbb{D} \rightarrow  \mathbb{C}$ by
\begin{equation}\label{ch3,eq6.2}
p(z):=\frac{zf'(z)}{f(z)}-\alpha,
\end{equation}
where $\alpha$ is the smallest positive root of \eqref{ch3,eq6.1}. Since   $f$ is a convex univalent function in $\mathcal{A}_b$, the function
\[p(z)=(1-\alpha)+bz+\cdots\]
is analytic in $\mathbb{D}$. Thus $p \in
\mathcal{H}_{b}[1-\alpha,1]$, and  $\alpha\leq2/3<1$ readily yields
\[ p(0)=1-\alpha >0.\]
From \eqref{ch3,eq6.2}, it follows that
\[\frac{zf'(z)}{f(z)}=p(z)+\alpha\]
so that
\begin{equation}\label{ch3,eq6.3}
\frac{zf''(z)}{f'(z)}+1=p(z)+\alpha+\frac{zp'(z)}{p(z)+\alpha}=\psi(p(z),zp'(z)),
\end{equation}
where
\[\psi(r,s)=r+\alpha+\frac{s}{r+\alpha}.\]

The function $\psi$  is continuous in the domain
$D=(\mathbb{C}\setminus \{-\alpha\})\times \mathbb{C}$,
$(1-\alpha,0) \in D$, and
  \[\RE \psi(1-\alpha,0)=1 >0.\]
We now show that the $\beta$-admissibility condition
\eqref{ch3,eq4.4} is satisfied.
  Since
  \[\psi(i\rho,\sigma)=i \rho+\alpha+\frac{\sigma}{\alpha^{2}+\rho^{2}}(\alpha-i\rho),\]
it follows that
  \begin{align*}
  \RE \psi(i\rho,\sigma)&=\alpha+\frac{\alpha \sigma}{\alpha^{2}+\rho^{2}}\\
                        & \leq \alpha-\frac{1}{2}\frac{\alpha}{\alpha^{2}+\rho^{2}}\left(1+\frac{2(1-\alpha)-|b|}{2(1-\alpha)+|b|}\right)\frac{(1-\alpha)^{2}+\rho^{2}}{1-\alpha}\\
                        &=\alpha-\frac{2\alpha}{2(1-\alpha)+|b|}\frac{(1-\alpha)^{2}+\rho^{2}}{\alpha^{2}+\rho^{2}}.
  \end{align*}
Since the function $h$ given by
  \[h(t)=\frac{(1-\alpha)^{2}+t}{\alpha^{2}+t},\quad t \geq0,\]
is an increasing function of $t$ for $\alpha\geq1/2$, clearly
  \[\frac{(1-\alpha)^{2}+\rho^{2}}{\alpha^{2}+\rho^{2}} \geq \frac{(1-\alpha)^{2}}{\alpha^{2}}.\]
  Thus, it follows from \eqref{ch3,eq6.1} that
  \begin{align*}
  \RE \psi(i\rho,\sigma)&\leq \alpha-\frac{2\alpha}{2(1-\alpha)+|b|}\frac{(1-\alpha)^{2}}{\alpha^{2}}\\
                        &=\frac{\alpha^{2}|b|-2\alpha^{3}+4\alpha-2}{\alpha[2(1-\alpha)+|b|]}=0,
  \end{align*}
whenever $\rho \in \mathbb{R}$ and
\[\sigma \leq -\frac{1}{2}\left(1+\frac{2\RE p(0)-|b|}{2\RE p(0)+|b|}\right)
\frac{|p(0)-i\rho|^{2}}{\RE p(0)},\quad p(0)=1-\alpha.\] Thus $\psi
\in \Psi_{b}(1-\alpha)$. The hypothesis and \eqref{ch3,eq6.3} give
\[\RE \psi(p(z),zp'(z)) >0 \quad ( z \in \mathbb{D}).\]
Therefore,  Theorem \ref{ch3,th4.8} (ii) shows that $p$ satisfies
$\RE p(z) >0$, and thus  $f\in \mathcal{S}^*(\alpha)$.
\end{proof}

\begin{remark}\label{ch3,rem6.2}
The roots of \eqref{ch3,eq6.1} in $[1/2,2/3]$ are decreasing as a
function of $|a_2|$, $|a_2|\in(0,1]$. If $|b|=|a_2|=1$, then
\eqref{ch3,eq6.1} becomes
\[2\alpha^{3}-\alpha^{2}-4\alpha+2=(\alpha^{2}-2)(2\alpha-1)=0.\]
Thus  $\alpha=1/2$. Therefore, Theorem \ref{ch3,th6.1} reduces to
Theorem 2.6a in  \cite[p.\ 57]{monograph} in this case, and provides
an improvement of the Marx-Strohh\"{a}cker's result that $f \in
\mathcal{C}$ implies $f \in \mathcal{S}^{*}(1/2)$.
\end{remark}

 \begin{theorem}\label{ch3,th6.3}
If $f =z+a_2z^2+\ldots\in \mathcal{C}$, then\[  \RE
\sqrt{f'(z)}>\alpha,\] where $\alpha$ is given by
\begin{equation}\label{ch3,eq6.4}
\alpha=\frac{10+|a_2|-\sqrt{|a_2|^{2}+20|a_2|+4}}{12}.
\end{equation}
\end{theorem}

\begin{proof}
Let $b=:a_2$. First note that $\alpha$ given by \eqref{ch3,eq6.4} satisfies the equation
\begin{equation}\label{ch3,eq6.5}
6\alpha^{2}-(10+|b|)\alpha+4=0.
\end{equation}
If $\alpha <1/2$, then
\[10+|b|-\sqrt{|b|^{2}+20|b|+4}<6\]
and leads to $|b|>1$, which  contradicts $f$ being convex. Similarly, if $\alpha >2/3$, then $|b|<0$ which is again a contradiction. Thus
$\alpha \in [1/2,2/3]$.

Define the function $p:\mathbb{D} \rightarrow  \mathbb{C}$ by
\begin{equation}\label{ch3,eq6.6}
p(z):=\sqrt{f'(z)}-\alpha.
\end{equation}
Since $f \in \A_{b}$ and $f$ is convex univalent, the function
\[p(z)=(1-\alpha)+bz+\cdots\]
is analytic in $\mathbb{D}$. Thus $p \in
\mathcal{H}_{b}[1-\alpha,1]$, and because  $\alpha\leq 2/3<1$, then
\[ p(0)=1-\alpha >0.\]
Now \eqref{ch3,eq6.6} yields
\[f'(z)=(p(z)+\alpha)^{2}\]
so that
\begin{equation}\label{ch3,eq6.7}
\frac{zf''(z)}{f'(z)}+1=\frac{2zp'(z)}{p(z)+\alpha}+1=\psi(p(z),zp'(z)),
\end{equation}
where
\[\psi(r,s)=\frac{2s}{r+\alpha}+1.\]

The function $\psi$ is continuous in the domain
$D=(\mathbb{C}\setminus\{-\alpha\})\times \mathbb{C}$,
 $(1-\alpha,0) \in D$ and $\RE \psi(1-\alpha,0)=1 >0.$
 To verify  the $\beta$-admissibility condition \eqref{ch3,eq4.4} is satisfied, consider
  \[\psi(i\rho,\sigma)=\frac{2\sigma}{\alpha^{2}+\rho^{2}}(\alpha-i\rho)+1.\]
Then
  \begin{align*}
  \RE \psi(i\rho,\sigma)&=\frac{2\alpha \sigma}{\alpha^{2}+\rho^{2}}+1\\
                        & \leq -\frac{\alpha}{\alpha^{2}+\rho^{2}}\left(1+\frac{2(1-\alpha)-|b|}{2(1-\alpha)+|b|}\right)\frac{(1-\alpha)^{2}+\rho^{2}}{1-\alpha}+1\\
                        &=-\frac{4\alpha}{2(1-\alpha)+|b|}\frac{(1-\alpha)^{2}+\rho^{2}}{\alpha^{2}+\rho^{2}}+1.
  \end{align*}
Using \eqref{ch3,eq6.5} and the monotonicity of the function
  \[h(t)=\frac{(1-\alpha)^{2}+t}{\alpha^{2}+t},\quad t \geq0,\]
it follows that
  \begin{align*}
     \RE \psi(i\rho,\sigma)&\leq -\frac{4\alpha}{2(1-\alpha)+|b|}\frac{(1-\alpha)^{2}}{\alpha^{2}}+1\\
                           &=\frac{(10+|b|)\alpha-6\alpha^{2}-4}{\alpha[2(1-\alpha)+|b|]}=0,
  \end{align*}
  whenever $\rho \in \mathbb{R}$ and
  \[\sigma \leq -\frac{1}{2}\left(1+\frac{2\RE p(0)-|b|}{2\RE p(0)+|b|}\right)\frac{|p(0)-i\rho|^{2}}{\RE p(0)},\quad p(0)=1-\alpha.\]
Thus $\psi \in \Psi_{b}(1-\alpha)$.

The hypothesis and \eqref{ch3,eq6.7}  yield
\[\RE \psi(p(z),zp'(z)) >0 \quad ( z\in
 {\mathbb{D}}).\]
Therefore, we conclude from Theorem \ref{ch3,th4.8}$(ii)$  that $p$
satisfies
\[\RE p(z) >0 \quad   \quad ( z \in \mathbb{D}).\]
This is equivalent to
\[\RE \sqrt{f'(z)}>\frac{10+|b|-\sqrt{|b|^{2}+20|b|+4}}{12}, \quad b=a_2. \qedhere\]
\end{proof}

\begin{remark}\label{ch3,rem6.4} The roots of \eqref{ch3,eq6.4} are
decreasing as a function of $|a_2|$, $0<|a_2|\leq1$. If $|a_2|=1$,
then $\alpha$ given by \eqref{ch3,eq6.4} reduces to $1/2$.
Therefore, Theorem \ref{ch3,th6.3} improves Theorem 2.6a in
\cite[p.\ 57]{monograph}.
\end{remark}

\begin{theorem}\label{ch3,th6.5}
If $f =z+a_2z^2+\ldots\in  \mathcal{S}^{*}$, then
\[ \RE \sqrt{\frac{f(z)}{z}}>\alpha,\]
where $\alpha$ is given by
\begin{equation}\label{ch3,eq6.8}
\alpha=\frac{20+|a_2|-\sqrt{16+|a_2|^{2}+40|a_2|}}{24}.
\end{equation}
\end{theorem}

\begin{proof} The result follows from Theorem~\ref{ch3,th6.3} and the fact that
the classes of convex and starlike functions functions are related
by the Alexander relation $f \in \mathcal{C}$ if and only if $zf' \in  \mathcal{S}^{*}$.%
\end{proof}

\begin{remark}\label{ch3,rem6.6} If $|a_2|=2$ then $\alpha$ given by
\eqref{ch3,eq6.8} reduces to $1/2$. Therefore, Theorem
\ref{ch3,th6.5} improves Theorem 2.6e in \cite[p.\ 62]{monograph}.\\
\end{remark}

\noindent \textbf{Acknowledgement.} The work presented here was
supported in parts by FRGS and RU grants from Universiti Sains
Malaysia. The authors are thankful to the referees for their useful
comments.

\end{document}